\title{\vspace{-1cm}Note on Lagrangian-Eulerian Methods for Uniqueness in Hydrodynamic Systems}
\date{\today}
\documentclass[11pt]{article}
\usepackage{authblk}
\usepackage{blindtext}
\usepackage{amsfonts}
\usepackage{amsmath}
\usepackage{amsthm}
\usepackage{graphicx}
\usepackage[margin=2cm]{geometry}
\newtheorem{theorem}{Theorem}
\newtheorem{lemma}{Lemma}

\newcommand{\pa}{\partial}
\newcommand{\la}{\label}
\newcommand{\fr}{\frac}
\newcommand{\na}{\nabla}
\newcommand{\be}{\begin{equation}}
\newcommand{\ee}{\end{equation}}
\newcommand{\ba}{\begin{array}{l}}
\newcommand{\ea}{\end{array}}
\newcommand{\Rr}{{\mathbb R}}
\newcommand{\beg}{\begin}

\newcommand{\norm}[1]{\left\lVert#1\right\rVert}
\author[1]{Peter Constantin \thanks{const@math.princeton.edu}}
\author[1]{	Joonhyun La \thanks{joonhyun@math.princeton.edu}}
\affil[1] {Department of mathematics, Princeton University}
\begin{document}
\maketitle
\abstract{We discuss the Lagrangian-Eulerian framework for hydrodynamic models and provide a proof of Lipschitz dependence of solutions on initial data in path space. The paper presents a corrected version of the result in \cite{c1}.}

\section{Introduction}
Many hydrodynamical systems consist of evolution equations for fluid velocities forced by external stresses, coupled to evolution equations for the external stresses. In the simplest cases, the Eulerian velocity $u$ can be recovered from the stresses $\sigma$ via a linear operator
\begin{equation}
u = \mathbb{U} (\sigma)
\end{equation}
and the stress matrix $\sigma$ obeys a transport and stretching equation of the form
\[
\pa_t \sigma + u\cdot\na\sigma = F(\na u, \sigma),
\] 
where $F$ is a nonlinear coupling depending on the model. The Eulerian velocity gradient is obtained in terms of the operator
\begin{equation}
\nabla_x u = \mathbb{G} (\sigma),
\end{equation}
and, in many cases, $\mathbb G$ is bounded in H\"{o}lder spaces of low regularity. Then, passing to Lagrangian variables, 
\[
\tau = \sigma \circ X
\]
where $X$ is the particle path transformation $X(\cdot, t) : \mathbb{R}^d \rightarrow \mathbb{R}^d $, a volume preserving diffeomorphism, the system becomes
\begin{equation}
\left \{
\begin{gathered}
\partial_t X = \mathcal{U} (X, \tau), \\
\partial_t \tau = \mathcal{T} (X, \tau).
\end{gathered}
\right.
\end{equation}
with
\begin{equation}
\begin{gathered}
\mathcal{U} (X, \tau) = \mathbb{U} (\tau \circ X^{-1} ) \circ X, \\
\mathcal{T} (X, \tau) = F(\mathbb{G}(\tau \circ X^{-1} ) \circ X, \tau).
\end{gathered}
\la{ugexp}
\end{equation}
In particular, $\tau$ solves an ODE
\begin{equation}
\frac{d}{dt} \tau = F(g, \tau)
\end{equation}
where $g = \nabla_x u \circ X$ is of the same order of magnitude as $\tau$ in appropriate spaces, and so the size of $\tau$ is readily estimated from the information provided by the ODE model, analysis of $\mathbb G$ and of the operation of composition with $X$. The main additional observation that leads to Lipschitz dependence in path space is that derivatives with respect to parameters of expressions of the type encountered in the Lagrangian evolution (\ref{ugexp}),
\[
   \mathbb{U} (\tau \circ X^{-1} ) \circ X, \quad \mathbb{G}(\tau \circ X^{-1} ) \circ X,
\]
introduce commutators, and these are well behaved in spaces of relatively low regularity. The Lagrangian-Eulerian method of \cite{MR3660694} formalized these considerations leading to uniqueness and Lipschitz dependence on initial data in path space, with application to several examples including incompressible  2D and 3D Euler equations, the surface quasi-geostrophic equation (SQG), the incompressible porous medium equation, the incompressible Boussinesq system, and the Oldroyd-B system coupled with the steady Stokes system. In all these examples the operators $\mathbb U$ and $\mathbb G$ are time-independent. 

The paper \cite{c1} considered time-dependent cases.  
When the operators $\mathbb U$ and $\mathbb G$ are time-dependent, in contrast to the time-independent cases studied in \cite{MR3660694},  $\mathbb G$ is not necessarily bounded in $L^{\infty}(0,T; C^{\alpha})$. This was addressed in \cite{c1} by using a H\"{o}lder continuity $\sigma\in C^{\beta}(0,T; C^{\alpha})$. While this treated the Eulerian issue, it was tacitly used but never explicitly stated in \cite{c1} that this kind of H\"{o}lder continuity is transferred to $\sigma$ from $\tau$ by composition with a smooth time-depending diffeomorphism close to the identity. This is false. In fact, we can easily give examples of $C^{\alpha}$ functions $\tau$ which are time-independent (hence analytic in time with values in $C^{\alpha}$) and diffeomorphisms $X(t)(a) = a +vt$ with constant $v$, such that $\sigma = \tau\circ X^{-1}$ is not continuous in $C^{\alpha}$ as a function of time.
In this paper we present a correct version of the results in \cite{c1}. Instead of relying on the time regularity of $\tau$ alone, we also use the fact that 
$\mathbb G$ is composed from a time-independent bounded operator and an operator whose kernel is smooth and rapidly decaying in space. Then the time singularity is resolved by using the Lipschitz dependence in $L^1$ of Schwartz functions composed with smoothly varying diffeomorphisms near the identity.

A typical example of the systems we can treat is the Oldroyd-B system coupled with Navier-Stokes equations:
\begin{equation}
\left \{
\begin{gathered}
 \partial_t u  - \nu \Delta u = \mathbb{H} \left ( \mathrm{div} \,\,(\sigma - u\otimes u) \right ), \\ \nabla \cdot u = 0, \\ \partial_t \sigma + u \cdot \nabla \sigma = (\nabla u) \sigma + \sigma (\nabla u)^T -2k\sigma + 2\rho K ((\nabla u) + (\nabla u)^T), \\ u(x,0) = u_0 (x), \sigma(x,0) = \sigma_0 (x).  \end{gathered} \right. \label{sys}
\end{equation}
Here $ (x, t) \in \mathbb{R}^d \times [0, T)$.  The Leray-Hodge projector  $\mathbb{H} = \mathbb{I} + R \otimes R$ is given in terms of the Riesz transforms $R = (R_1,\dots, R_d)$,  and $\nu, \rho K, k$ are fixed positive constants. This system is viscoelastic, and  the behavior of the solution depends on the history of its deformation. \newline
The non-resistive MHD system 
\begin{equation}
\left \{
\begin{gathered}
 \partial_t u  - \nu \Delta u = \mathbb{H} \left ( \mathrm{div} \,\,(b\otimes b - u\otimes u) \right ), \\ \nabla \cdot u = 0, \\\nabla\cdot b = 0,\\ \partial_t b + u \cdot \nabla b = (\nabla u) b, \\ u(x,0) = u_0 (x), b(x,0) = b_0 (x).  \end{gathered} \right. \label{sysb}
\end{equation}
can also be treated by this method. The systems (\ref{sys}) and (\ref{sysb}) have been studied extensively, and a review of the literature is beyond the scope of this paper.


\section{The Lagrangian-Eulerian formulation} \label{Formulation}
We show calculations for (\ref{sys}) in order to be explicit, and because the calculations for (\ref{sysb}) are entirely similar.
The solution map for $u(x,t)$ of (\ref{sys}) is
\begin{equation}
\begin{gathered}
u(x,t) = \mathbb{L}_\nu (u_0) (x, t) + \int_0 ^t  g_{\nu(t-s)} * \left ( \mathbb{H} \left ( \mathrm{div} \,\, (\sigma - u \otimes u) \right ) \right )  (x,s) ds .
\end{gathered} 
\end{equation}
where 
\begin{equation}
\mathbb{L}_\nu (u_0 ) (x, t) = g_{\nu t} * u_0 (x) = \int_{\mathbb{R}^d} \frac{1}{(4\pi \nu t) ^{\frac{d}{2}} } e^{-\frac{|x-y|^2}{4\nu t}} u_0 (y) dy.
\end{equation}
Thoroughout the paper we use
\[
g_{\nu t}(x) = \frac{1}{(4\pi \nu t) ^{\frac{d}{2}} } e^{-\frac{|x|^2}{4\nu t}}.
\]
The velocity gradient satisfies
\begin{equation}
(\nabla u) (x,t) = \mathbb{L}_\nu (\nabla u_0) (x, t) + \int_0 ^t \left ( g_{\nu (t-s)} * \left ( \mathbb{H}  \nabla \mathrm{div} \,\, (\sigma - u \otimes u)  \right ) \right ) (x,s) ds.
\end{equation}
We denote the Eulerian velocity and gradient operators 
\begin{equation}
\left \{
\begin{gathered}
\mathbb{U} (f) (x,t) = \int_0 ^t ( g_{\nu(t-s)} * \mathbb{H} \mathrm{div} \,\, f ) (x,s) ds, \\
\mathbb{G} (f) (x,t) = \int_0 ^t ( g_{\nu(t-s)} * \mathbb{H} \nabla \mathrm{div} \,\, f ) (x,s) ds.
\end{gathered} \right.
\end{equation}
Note that for a second order tensor $f$, $\mathbb{G} (f) = \nabla_x \mathbb{U} (f) =  R \otimes R \left ( \mathbb{U} (\nabla_x f) \right )$. Let $X$ be the Lagrangian path diffeomorphism, $v$ the Lagrangian velocity, and $\tau $ the Lagrangian added stress, 
\begin{equation}
\begin{gathered}
v = \frac{\partial X}{\partial t} = u \circ X, \\
\tau = \sigma \circ X. \label{lagVar}
\end{gathered}
\end{equation}
We also set
\begin{equation}
\begin{gathered}
g (a,t) = (\nabla u) (X(a,t), t) =   \mathbb{L}_\nu (\nabla u_0) \circ X (a,t) \\ + \mathbb{G} \left (\tau \circ X^{-1} \right ) \circ X (a,t) - \mathbb{U} \left (\nabla_x \left ( ( v \otimes v ) \circ X^{-1} \right ) \right ) \circ X  (a,t).
\end{gathered}
\label{gat}
\end{equation}
In Lagrangian variables the system is
\begin{equation}
\left \{ 
\begin{gathered}
X(a,t) = a + \int_0 ^t \mathcal{V} (X,\tau, a,s) ds, \\
\tau(a,t) = \sigma_0 (a) + \int_0 ^t \mathcal{T} (X, \tau,a,s) ds , \\
v(a,t) = \mathcal{V} (X, \tau, t)
\end{gathered} \right. \label{fixedpt}
\end{equation}
where the Lagrangian nonlinearities $\mathcal{V}, \mathcal{T}$ are
\begin{equation}
\left \{ \begin{gathered}
\mathcal{V} (X, \tau, a,s) =  \mathbb{L}_\nu ( u_0) \circ X  (a,s) +  (\mathbb{U} \left (\left ( \tau - v \otimes v) \circ X^{-1} \right ) \right ) \circ X  (a,s), \\ \mathcal{T} (X, \tau, a,s) = \left ( g \tau + \tau g^T - 2k \tau + 2\rho K (g + g^T ) \right)  (a,s),
\end{gathered} \right.
\end{equation}
and  $g$ is defined above in (\ref{gat}).  The main result of the paper is
\begin{theorem}\la{main}
Let $0<\alpha<1$ and $1<p<\infty$, be given. Let also $v_1(0) = u_1(0) \in C^{1+\alpha,p}$ and $v_2(0) = u_2(0)\in C^{1+\alpha, p}$ be given divergence-free initial velocities, and $\sigma_1(0), \sigma_2(0)\in C^{\alpha,p}$ be given initial stresses. Then there exists $T_0>0$ and $C>0$ depending on the norms of the initial data such that  
 $(X_1, \tau_1, v_1), (X_2, \tau_2, v_2)$, with initial data $(Id, \sigma_1 (0), u_1 (0)), (Id, \sigma_2 (0), u_2 (0))$, are bounded in 
$Id +Lip(0,T_0; C^{1+\alpha, p})\times Lip (0,T_0; C^{\alpha,p})\times L^{\infty}(0,T_0; C^{1+\alpha,p})$ and solve the Lagrangian form (\ref{fixedpt}) of (\ref{sys}). Moreover, 
\begin{equation}
\begin{gathered}
\norm{X_2 - X_1} _{Lip(0, T_0; C^{1+\alpha, p})} + \norm{\tau_2 - \tau_1 }_{Lip(0, T_0; C^{\alpha, p})} + \norm{v_2 - v_1}_{L^\infty (0, T_0, C^{1+\alpha, p})} 
\\ \le C ( \norm{u_2 (0) - u_1 (0) }_{1+\alpha, p} + \norm{\tau_2 (0) - \tau_1 (0) } _{\alpha, p})
\end{gathered}
\end{equation} \label{Uniq}
\end{theorem}
\beg{remark}
The solutions' Lagrangian stresses $\tau$ are Lipschitz in time with values in $C^{\alpha}$. Their Lagrangian counterparts $\sigma = \tau\circ X^{-1}$ are  bounded in time with values in $C^{\alpha}$ and space-time H\"{o}lder continuous with exponent $\alpha$. The Eulerian version of the equations (\ref{sys}) is satisfied in the sense of distributions, and solutions are unique in this class.
\end{remark}

The spaces $C^{\alpha, p}$ are defined in the next section. The proof of the theorem occupies the rest of the paper. We start by considering variations of Lagrangian variables. We take a family $(X_\epsilon, \tau_\epsilon )$ of flow maps depending smoothly on a parameter $\epsilon \in [1,2]$, with initial data $u_{\epsilon, 0}$ and $\sigma_{\epsilon, 0}$. Note that $v_\epsilon = \partial_t X_\epsilon$. We use the following notations
\begin{equation}
\left \{ \begin{gathered}
u_\epsilon = \partial_t X_\epsilon \circ X_\epsilon ^{-1}, g_\epsilon ' = \frac{d}{d\epsilon} g_\epsilon, \\
X_\epsilon ' = \frac{d}{d\epsilon} X_\epsilon, \, \eta_\epsilon = X_\epsilon ' \circ X_\epsilon ^{-1}, \\
v_\epsilon ' = \frac{d}{d\epsilon} v_\epsilon, \\
\sigma_\epsilon = \tau_\epsilon \circ X_\epsilon^{-1}, \\
\tau_\epsilon ' = \frac{d}{d\epsilon} \tau_\epsilon, \delta_\epsilon = \tau_\epsilon ' \circ X_\epsilon ^{-1}, 
\end{gathered} \right.
\end{equation} 
and 
\begin{equation}
u_{\epsilon, 0} ' = \frac{d}{d\epsilon} u_\epsilon (0), \sigma_{\epsilon, 0} ' = \frac{d}{d\epsilon} \sigma_{\epsilon } (0).
\end{equation}
We  represent
\begin{equation}
\left \{ \begin{gathered}
X_2 (a,t) - X_1 (a,t) = \int_1 ^2 \mathcal{X}_\epsilon '  d\epsilon, \\
\tau_2 (a,t) - \tau_1 (a,t) =  \int_1 ^2  \pi_\epsilon d\epsilon  , \\
v_2 (a,t) - v_1 (a,t) = \int_1 ^2 \frac{d}{d\epsilon} \mathcal{V}_\epsilon d \epsilon,
\end{gathered} \right.
\end{equation}
where
\begin{equation}
\begin{gathered}
\mathcal{X}_\epsilon ' = \int_0 ^t \frac{d}{d\epsilon} \mathcal{V}_\epsilon  ds,\,\,  \pi_\epsilon = \int_0 ^t \frac{d}{d\epsilon} \mathcal{T}_\epsilon ds + \sigma_{\epsilon, 0} ', \\
\mathcal{V}_\epsilon = \mathcal{V} (X_\epsilon, \tau_\epsilon), \, \mathcal{T}_\epsilon = \mathcal{T} (X_\epsilon, \tau_\epsilon). \label{vars}
\end{gathered}
\end{equation}

We have the following commutator expressions arising by differentiating in $\epsilon$ (\cite{c1}, \cite{MR3660694})):

\begin{equation}
\begin{gathered}
\left ( \frac{d}{d\epsilon} \left ( \mathbb{U} (\tau_\epsilon \circ X_\epsilon ^{-1} ) \circ X_\epsilon \right ) \right ) \circ X_\epsilon ^{-1} = [\eta_\epsilon \cdot \nabla_x, \mathbb{U} ] (\sigma_\epsilon) + \mathbb{U} (\delta_\epsilon), 
\end{gathered}
\end{equation}
where 
\begin{equation}
[\eta_\epsilon \cdot \nabla_x, \mathbb{U} ] (\sigma_\epsilon) = \eta_\epsilon \cdot \nabla_x \left ( \mathbb{U} (\sigma_\epsilon) \right ) - \mathbb{U} \left ( \eta_\epsilon \cdot \nabla_x \sigma_\epsilon \right )
\end{equation}
and 
\begin{equation}
\begin{gathered}
\left (\frac{d} {d\epsilon} \mathbb{U} (v_\epsilon \otimes v_\epsilon \circ X_\epsilon ^{-1} ) \circ X_\epsilon \right ) \circ X_\epsilon ^{-1} \\
 = [ \eta_\epsilon  \cdot  \nabla_x,  \mathbb{U} ] (u_\epsilon \otimes u_\epsilon)    +  \mathbb{U} ((v_\epsilon' \otimes v_\epsilon + v_\epsilon \otimes v_\epsilon ' )\circ X_\epsilon ^{-1} ) .
\end{gathered}
\end{equation}
We note, by the chain rule,
\begin{equation}
\nabla_a \mathcal{V} = \left ( \nabla_a X \right ) g. \label{Vg}
\end{equation}
Consequently, differentiating $\mathcal V_{\epsilon}$, $g_{\epsilon}$ and the relation (\ref{Vg}) we have
\begin{equation}
\left \{
\begin{gathered}
\left ( \frac{d}{d\epsilon} \mathcal{V}_\epsilon \right ) \circ X_\epsilon ^{-1} = \eta_\epsilon \cdot ( \mathbb{L}_\nu (\nabla_x u_{\epsilon, 0})) + \mathbb{L}_\nu ( u_{\epsilon, 0} ' ) \\ + [\eta_\epsilon \cdot \nabla_x, \mathbb{U}] (\sigma_\epsilon - u_\epsilon \otimes u_\epsilon ) + \mathbb{U} (\delta_\epsilon - (v_\epsilon ' \otimes v_\epsilon + v_\epsilon \otimes v_\epsilon ' ) \circ X_\epsilon ^{-1} ), \\
g_\epsilon = \mathbb{L} (\nabla_x u_{\epsilon, 0} ) \circ X_\epsilon + \mathbb{G} (\sigma_\epsilon ) \circ X_\epsilon - \mathbb{U} (\nabla_x (u_\epsilon \otimes u_\epsilon ) ) \circ X_\epsilon, \\
g_\epsilon ' \circ X_\epsilon ^{-1} = \eta_\epsilon  \cdot \mathbb{L}_\nu (\nabla_x \nabla_x u_{\epsilon,0} ) + \mathbb{L}_\nu ( \nabla_x u_{\epsilon, 0} ') + [\eta_\epsilon \cdot \nabla_x, \mathbb{G} ] (\sigma_\epsilon ) + \mathbb{G} (\delta_\epsilon ) \\
- [\eta_\epsilon \cdot \nabla_x, \mathbb{U} ] \left ( \nabla_x (u_\epsilon \otimes u_\epsilon ) \right ) - \mathbb{U} \left (  \nabla_x \left ( \left (v_\epsilon ' \otimes v_\epsilon + v_\epsilon \otimes v_\epsilon ' \right ) \circ X_\epsilon ^{-1} \right ) \right ), \\
\frac{d}{d\epsilon} (\nabla_a \mathcal{V}_\epsilon ) = (\nabla_a X_\epsilon ' ) g_\epsilon + ( \nabla_a X_\epsilon ) g_\epsilon ', \\
\frac{d}{d\epsilon} \mathcal{T}_\epsilon = g_\epsilon ' \tau_\epsilon + g_\epsilon \tau_\epsilon' + \tau_\epsilon ' g_\epsilon ^T + \tau_\epsilon (g_\epsilon ' )^T - 2k\tau_\epsilon ' + 2\rho K (g_\epsilon ' + (g_\epsilon ' )^T). \\
\end{gathered} \right.
\end{equation}
\section{Functions, operators, commutators} \label{Operators}
We consider function spaces
\begin{equation}
C^{\alpha, p} = C^{\alpha} (\mathbb{R}^d) \cap L^p (\mathbb{R}^d)
\end{equation}
with norm
\begin{equation}
\norm{f}_{\alpha, p} = \norm{f}_{C^{\alpha} (\mathbb{R}^d)} + \norm{f}_{L^p (\mathbb{R}^d )}
\end{equation}
for $\alpha \in (0,1), p \in (1, \infty )$, $C^{1+\alpha} (\mathbb{R}^d)$ with norm
\begin{equation}
\norm{f}_{C^{1+\alpha} (\mathbb{R}^d)} = \norm{f}_{L^\infty (\mathbb{R}^d) } + \norm{ \nabla f} _{C^{\alpha} (\mathbb{R}^d)},
\end{equation}
and 
\begin{equation}
C^{1+\alpha, p} = C^{1+\alpha} (\mathbb{R}^d) \cap W^{1, p} (\mathbb{R}^d)
\end{equation}
with norm
\begin{equation}
\norm{f}_{1+\alpha, p} = \norm{f}_{C^{1+\alpha} (\mathbb{R}^d) } + \norm{f}_{W^{1,p} (\mathbb{R}^d ) } .
\end{equation}
We also use spaces of paths, $L^\infty(0, T; Y)$ with the usual norm,
\begin{equation}
\norm{f}_{L^\infty (0, T; Y) } = \sup_{t \in [0, T] } \norm{f(t)}_Y,
\end{equation}
spaces $Lip(0, T; Y)$ with norm
\begin{equation}
\norm{f}_{Lip(0, T; Y)} = \sup_{t\ne s, t,s \in [0, T] } \frac{ \norm{f(t) - f(s) }_Y}{|t-s|} + \norm{f}_{L^\infty (0, T; Y)}
\end{equation}
where $Y$ is $C^{\alpha, p}$ or $C^{1+\alpha, p}$ in the following. We use the following lemmas.
\begin{lemma} [ \cite{MR3660694}] \label{ComEs1}
Let $0 < \alpha < 1$, $1 < p < \infty$. Let $\eta \in C^{1+\alpha}(\mathbb{R}^d)$ and let 
\begin{equation}
(\mathbb{K} \sigma ) (x) = P. V. \int_{\mathbb{R}^d} k (x-y)\sigma(y)dy 
\end{equation}
be a classical Calderon-Zygmund operator with kernel $k$ which is smooth away from the origin, homogeneous of degree $-d$ and with mean zero on spheres about the origin. Then the commutator $[\eta \cdot \nabla, \mathbb{K}]$ can be defined as a bounded linear operator in $C^{\alpha, p}$ and
\begin{equation}
\norm{[\eta \cdot \nabla, \mathbb{K} ] \sigma}_{C^{\alpha, p} } \le C \norm{\eta}_{C^{1+\alpha}(\mathbb{R}^d)} \norm{\sigma}_{C^{\alpha, p} }.
\end{equation}
\end{lemma}
\begin{lemma} [Generalized Young's inequality] \label{GenY}
Let $1 \le q \le \infty$ and $C > 0$. Suppose $K$ is a measurable function on $\mathbb{R}^d \times \mathbb{R}^d$ such that
\begin{equation}
\sup_{x \in \mathbb{R}^d} \int_{\mathbb{R}^d} | K(x, y) | dy \le C, \,\, \sup_{y \in \mathbb{R}^d} \int_{\mathbb{R}^d } |K(x, y) | dx \le C.
\end{equation}
If $f \in L^q (\mathbb{R}^d)$, the function $Tf$ defined by
\begin{equation}
Tf (x) = \int_{\mathbb{R}^d} K(x, y) f(y) dy
\end{equation}
is well defined almost everywhere and is in $L^q$, and $\norm{Tf}_{L^q} \le C \norm{f}_{L^q}$.
\end{lemma}
The proof of this lemma for $1<q<\infty$ is done using duality, a straightforward application of Young's inequality and changing order of integration. The extreme cases $q=1$ and $q=\infty$ are proved directly by inspection.

For simplicity of notation, let us denote
\begin{equation}
M_X = 1 + \norm{X - \mathrm{Id}}_{L^\infty (0, T; C^{1+\alpha})}.
\end{equation}
\begin{theorem}
Let $0 < \alpha < 1, 1<p<\infty$ and let $T>0$. Also let $X$ be a volume preserving diffeomorphism such that $X - \mathrm{Id} \in Lip (0, T; C^{1+\alpha})$. Then
\begin{equation}
\begin{gathered}
\norm{\tau \circ X^{-1} }_{L^\infty (0, T; C^{\alpha, p})} \le \norm{\tau}_{L^\infty (0, T; C^{\alpha, p})} M_X ^\alpha.
\end{gathered}
\end{equation} 
If $X' \in Lip(0, T; C^{1+\alpha})$, then
\begin{equation}
\norm{X' \circ X^{-1} }_{L^\infty (0, T; C^{1+\alpha})} \le \norm{X'}_{L^\infty (0, T; C^{1+\alpha})} M_X ^{1+ 2\alpha}.
\end{equation}
If $v \in Lip(0, T; W^{1,p})$, then
\begin{equation}
\norm{v \circ X^{-1} }_{L^\infty (0, T; W^{1,p} ) } \le \norm{v}_{L^\infty (0, T; W^{1,p})} M_X .
\end{equation}
If in addition $\partial_t X'$, $\partial_t X $ exist in $L^\infty (0, T; C^{1+\alpha})$, then
\begin{equation}
\begin{gathered}
\norm{X' \circ X^{-1} }_{Lip(0, T; C^{\alpha} ) } \le \norm{X'}_{Lip(0, T; C^{1+\alpha})} \norm{X - \mathrm{Id}}_{Lip(0, T; C^{1+\alpha})}  M_X ^{1+ 3\alpha}.
\end{gathered}
\end{equation}
\label{p}
\end{theorem}
\begin{proof}
\begin{equation}
\norm{\tau \circ X^{-1}}_{L^p \cap L^\infty} = \norm{\tau}_{L^p \cap L^\infty},
\end{equation}
and, denoting the seminorm
\[
\left [\tau\right]_{\alpha} = \sup_{a\ne b, a,b\in \Rr^2}\fr{|\tau(a)-\tau(b)|}{|a-b|^{\alpha}}
\]
we have
\begin{equation}
\begin{gathered}
\left [ \tau \circ X^{-1} (t) \right ]_\alpha \le \left [ \tau (t) \right ]_\alpha \norm{\nabla_x X^{-1} (t) } _{L^\infty} ^{\alpha} \le \left [ \tau (t) \right ]_\alpha  (1+ \norm{X - \mathrm{Id}}_{L^\infty (0, T; C^{1+\alpha})}  ) ^\alpha .
\end{gathered}
\end{equation}
Note that this shows that the same bound holds when we replace $X^{-1}$ by $X$. For the second and third part, it suffices to remark that
\begin{equation}
\nabla_x (X' \circ X^{-1} ) =  \left ( (\nabla_a X ) \circ X^{-1} \right ) ^{-1}  \left ( (\nabla_a X') \circ X^{-1} \right ) 
\end{equation}
and the previous part gives the bound in terms of Lagrangian variables. For the last part, we note that
\begin{equation}
\begin{gathered}
\frac{1}{t-s} \left ( X' \left (X^{-1} (x,t),t \right ) - X' \left (X^{-1} (x,s),s \right ) \right ) \\  = \int_0 ^1 \left ( (\partial_t X' ) \left (X^{-1} (x,\beta_\tau),\beta_\tau \right ) + \left (\partial_t X^{-1} \right )(x, \beta_\tau) (\nabla_a X')\left (X^{-1} (x, \beta_\tau),\beta_\tau \right ) \right ) d\tau, 
\end{gathered}
\end{equation}
where
\begin{equation}
\beta_\tau = \tau t + (1- \tau) s.
\end{equation}
Now noting that
\begin{equation}
\partial_t X^{-1} = - \left (\left (\partial_t X \right ) \circ X^{-1} \right ) \left ( (\nabla_a X)^{-1} \circ X^{-1}  \right )
\end{equation}
we have
\begin{equation}
\begin{gathered}
\frac{1}{t-s} \norm{X' \circ X^{-1} (t) - X' \circ X^{-1} (s) }_{C^{\alpha}} \\ 
\le \left ( \norm{\partial_t X' }_{L^\infty (0, T; C^{\alpha})} + \norm{\partial_t X}_{L^\infty (0, T; C^{\alpha} ) } \norm{X'}_{L^\infty (0, T; C^{1+\alpha} )} \right )  \left (1 + \norm{X - \mathrm{Id}}_{L^\infty (0, T; C^{1+\alpha})} \right ) ^{1+ 3\alpha}
\end{gathered}
\end{equation}
so that
\begin{equation}
\begin{gathered}
\norm{X' \circ X^{-1} }_{Lip(0, T; C^{\alpha} ) } \le \norm{X'}_{Lip(0, T; C^{1+\alpha})} \norm{X - \mathrm{Id}}_{Lip(0, T; C^{1+\alpha})} \left (1 + \norm{X - \mathrm{Id}}_{L^\infty (0, T; C^{1+\alpha})} \right ) ^{1+ 3\alpha}.
\end{gathered}
\end{equation}
\end{proof}

\begin{theorem}
Let $0 < \alpha < 1, 1 < p < \infty$ and let $T>0$. There exists a constant $C$ independent of $T$ and $\nu$ such that for any $0< t< T$,
\begin{equation}
\begin{gathered}
\norm{\mathbb{L}_\nu (u_0) }_{L^\infty (0, T; C^{\alpha, p})} \le C \norm{u_0}_{\alpha, p}, \\
\norm{\mathbb{L}_\nu (u_0) }_{L^\infty (0, T; C^{1+\alpha, p})} \le C \norm{u_0}_{1+ \alpha, p}, \\
\norm{\mathbb{L}_\nu (\nabla u_0 ) ( t) }_{\alpha, p} \le \frac{C}{(\nu t)^{\frac{1}{2}}} \norm{u_0}_{\alpha, p}, \\
\norm{\mathbb{L}_\nu (\nabla u_0)}_{L^\infty (0, T; C^{\alpha, p})} \le C \norm{u_0}_{1+\alpha, p}
\end{gathered}
\end{equation} \label{L}
hold.
\end{theorem}
\begin{proof}
\begin{equation}
\begin{gathered}
\norm{\mathbb{L}_\nu (u_0) (t) } _{\alpha, p} \le \norm{g_{\nu t} }_{L^1} \norm{u_0}_{\alpha, p} = \norm{u_0} _{\alpha, p}, \\
\norm{\mathbb{L}_\nu (u_0) (t) } _{1+ \alpha, p} \le \norm{g_{\nu t} }_{L^1} \norm{u_0}_{1+\alpha, p} = \norm{u_0} _{1+\alpha, p}, \\
\norm{\mathbb{L}_\nu (\nabla u_0) (t) } _{\alpha, p} \le \norm{\nabla g_{\nu t} }_{L^1} \norm{u_0}_{1+ \alpha, p} = \frac{C}{(\nu t)^{\frac{1}{2}}} \norm{u_0} _{\alpha, p}, \\
\norm{\mathbb{L}_\nu (\nabla u_0) (t) } _{\alpha, p} \le \norm{g_{\nu t} }_{L^1} \norm{\nabla u_0}_{\alpha, p} \le \norm{u_0} _{1+ \alpha, p}.
\end{gathered}
\end{equation}
\end{proof}
\begin{theorem}
Let $0<\alpha<1, 1<p<\infty$ and let $T>0$. There exists a constant $C$ such that
\begin{equation}
\norm{\mathbb{U} (\sigma) } _{L^\infty (0, T; C^{\alpha, p})} \le C \left ( \frac{T}{\nu} \right )^{\frac{1}{2}} \norm{\sigma}_{L^\infty (0, T; C^{\alpha, p})}.
\end{equation} \label{U}
\end{theorem}
\begin{proof}
\begin{equation}
\begin{gathered}
\norm{\mathbb{U} (\sigma) (t)}_{C^{\alpha, p}} \le C \int_0 ^t \norm{\nabla g_{\nu (t-s)}}_{L^1} \norm{\sigma(s) }_{\alpha, p} ds \\
\le \frac{C}{\nu^{\frac{1}{2}}} \int_0 ^t \frac{1}{(t-s)^{\frac{1}{2}}} ds \norm{\sigma}_{L^\infty (0, T; C^{\alpha, p})} \le \frac{C}{\nu^{\frac{1}{2}}} \sqrt{T} \norm{\sigma}_{L^\infty (0, T; C^{\alpha, p})}. 
\end{gathered}
\end{equation}
\end{proof}
\begin{theorem}
Let $0<\alpha<1, 1<p<\infty$ and let $T>0$. There exist constants $C_1, C_2$ depending only on $\alpha$ and $\nu$, and $C_3 (T, X), C_4 (T,X)$ such that
\begin{equation}
\begin{gathered}
\norm{\mathbb{G}(\tau \circ X^{-1} )  }_{L^\infty (0, T; C^{\alpha, p})} \le C_1 \norm{X - \mathrm{Id}}_{Lip(0, T; C^{1+\alpha})} ^\alpha \norm{\tau(0)}_{\alpha, p} (1 + C_3 (T, X)) 
\\+ C_2 \norm{\tau}_{Lip(0, T; C^{\alpha, p})} C_4 (T,X)
\end{gathered}
\end{equation} 
where $C_3 (T,X)$ and $C_4 (T,X)$ are of the form $C T^{\frac{1}{2}} \left (\norm{X - \mathrm{Id}}_{Lip(0, T; C^{1+\alpha})}^{\alpha} + \norm{X - \mathrm{Id}}_{Lip(0, T; C^{1+\alpha})}^4 \right ) $. 
\label{G}
\end{theorem}
\begin{proof}
Since $\mathbb{G} = (R \otimes R) \mathbb{H} \Gamma$ where
\begin{equation}
\Gamma (\tau \circ X^{-1} ) = \int_0 ^t \Delta g_{\nu (t-s) } * (\tau \circ X^{-1} (s) ) ds, \label{Gamma}
\end{equation}
we can replace $\mathbb{G}$ by $\Gamma$. Then $\Gamma (\tau \circ X^{-1} )$ can be written as
\begin{equation}
\begin{gathered}
\Gamma (\tau \circ X^{-1}) (t) = \int_0 ^t \Delta g_{\nu (t-s)} * \left ( \left ( \tau \circ X^{-1} \right ) (s) - \left (\tau \circ X^{-1} \right ) (t) \right ) ds \\ + \int_0 ^t \Delta g_{\nu (t-s) } * \left (\tau \circ X^{-1} \right ) (t) ds.
\end{gathered}
\end{equation}
But
\begin{equation}
\begin{gathered}
\int_0 ^t \Delta g_{\nu (t-s) }*( \tau \circ X^{-1} )(t) ds = \tau \circ X^{-1} (t) - g_{\nu t} * (\tau \circ X^{-1}) (t)
\end{gathered}
\end{equation}
so the second term is bounded by $2\norm{\tau}_{L^\infty (0, T; C^{\alpha, p})} M_X ^{\alpha} $ by Theorem \ref{p}. Now we let
\begin{equation}
\tau \circ X^{-1} (x,s) - \tau \circ X^{-1} (x,t) = \Delta_1 \tau (x,s,t) + \Delta_2 \tau (x,s,t),
\end{equation}
where
\begin{equation}
\begin{gathered}
\Delta_1 \tau (x,s,t) = \tau (X^{-1} (x,s), s) - \tau (X^{-1} (x,s), t), \\
\Delta_2 \tau (x,s,t) = \tau (X^{-1} (x,s), t) - \tau (X^{-1} (x,t), t).
\end{gathered} \label{Deltas}
\end{equation}
But since 
\begin{equation}
\begin{gathered}
\norm{\Delta_1 \tau (s,t) }_{C^{\alpha, p}}  \le |t-s| M_X ^{\alpha} \norm{\tau}_{Lip(0, T; C^{\alpha, p})} , \\
\end{gathered}
\end{equation}
by the proof of Theorem \ref{p} we get
\begin{equation}
\begin{gathered}
\norm{ \int_0 ^t \Delta g_{\nu (t-s)} * \Delta_1 \tau (s,t) ds }_{\alpha, p} \le \frac{Ct}{\nu}  \norm{\tau}_{Lip(0, T; C^{\alpha, p})} M_X ^{\alpha}, \\
\end{gathered} 
\end{equation}
On the other hand,
\begin{equation}
\int_0 ^t \Delta g_{\nu (t-s)}* \Delta_2 \tau (s,t) ds = \int_ 0 ^t \int_{\mathbb{R}^d} K(x,z,t,s) \tau(z,t) dz ds,
\end{equation}
where
\begin{equation}
K(x,z,t,s) = \Delta g_{\nu (t-s)} (x - X(z,s)) - \Delta g_{\nu (t-s) } (x - X(z,t)). \label{K}
\end{equation}
We use the following lemma.
\begin{lemma} \label{Ktype}
$K(x,z,t,s)$ is $L^1$ in both the $x$ variable and the $z$ variable, and
\begin{equation}
\sup_{z } \norm{K(\cdot, z, t, s )}_{L^1 }, \sup_{x} \norm{K (x, \cdot, t,s ) } _{L^1} \le  \frac{C \norm{X - \mathrm{Id}}_{Lip(0, T; L^\infty)}  } {|t-s|^{\frac{1}{2}} \nu^{\frac{3}{2}}}.
\end{equation}
\end{lemma}
\begin{proof}
We define
\begin{equation}
S(x) = 4 \pi e^{-|x|^2} \left ( |x|^2 -\frac{d}{2} \right ) 
\end{equation}
so that
\begin{equation}
(\Delta g_{\nu (t-s)} ) =  (4\pi\nu(t-s))^{-(\frac{d}{2}+1)} S \left (\frac{x}{(4(t-s))^{\frac{1}{2}}} \right ).
\end{equation}
Then
\begin{equation}
\begin{gathered}
\int |K(x, z, t, s)|dz = \int (4\pi \nu (t-s))^{-\left (\frac{d}{2} + 1 \right )} \left |  S(\frac{x-X(z,s)}{(4 \nu (t-s))^{\frac{1}{2}}}) - S(\frac{x-X(z,t)}{(4 \nu (t-s))^{\frac{1}{2}}})  \right | dz \\
= \int (4\pi \nu (t-s))^{- \left ( \frac{d}{2} + 1 \right ) } \left |  S(\frac{x-y}{(4 \nu(t-s))^{\frac{1}{2}}}) - S(\frac{x-X(y,t-s)}{(4 \nu (t-s))^{\frac{1}{2}}})  \right | dy \\
= \left ( 4 \pi \nu (t-s) \right )^{-1} \pi^{ - \left ( \frac{d}{2} + 1\right )} \int \left |  S(u) - S \left ( u - \frac{(X - \mathrm{Id}) (x - (4(t-s))^{\frac{1}{2}}u, t-s)}{(4 \nu(t-s))^{\frac{1}{2}}} \right )  \right | du.
\end{gathered}
\end{equation}
However, for each $u$
\begin{equation}
\begin{gathered}
\left |  S(u) - S \left (u - \frac{(X - \mathrm{Id}) (x - (4 \nu(t-s))^{\frac{1}{2}}u, t-s)}{(4 \nu (t-s))^{\frac{1}{2}}} \right )  \right |  \le \left | \frac{(X- \mathrm{Id})(x- (4 \nu(t-s))^{\frac{1}{2}}u, t-s)}{(4 \nu (t-s))^{\frac{1}{2}}} \right | \\ \times \sup  \left \lbrace |\nabla S (u - z) | : {|z| \le \left | \frac{(X- \mathrm{Id})(x- (4 \nu (t-s))^{\frac{1}{2}}u, t-s)}{(4 \nu (t-s))^{\frac{1}{2}}} \right |} \right  \rbrace
\end{gathered}
\end{equation}
and we have
\begin{equation}
\left | \frac{(X- \mathrm{Id})(x- (4 \nu (t-s))^{\frac{1}{2}}u, t-s)}{(4 \nu (t-s))^{\frac{1}{2}}} \right | \le \norm{(X-\mathrm{Id})}_{Lip(0, T; L^\infty)} \frac{|t-s|^{\frac{1}{2}}}{\nu^{\frac{1}{2} } } \le CT^{\frac{1}{2}}
\end{equation}
and obviously 
\begin{equation}
\tilde{S} (u) = \sup_{z \le C T ^{\frac{1}{2}}} |(\nabla S) ( u-z)|
\end{equation}
is integrable in $\mathbb{R}^d$; because $\nabla S$ is Schwartz, 
\begin{equation}
\begin{gathered}
|(\nabla S) (x)| \le \frac{C_d}{(1+ 2C^2 T + |x|^2 )^d}
\end{gathered}
\end{equation}
for some constant $C_d$, but if $|z| \le CT^{\frac{1}{2}}$, then $|u-z|^2 \ge |u|^2 - C^2 T$ and
\begin{equation}
|(\nabla S)(u-z)| \le \frac{C_d}{(1+C^2 T + |u|^2)^d}
\end{equation}
and the right side of above is clearly integrable with bound depending only on $d$ and $T$. Therefore, we have
\begin{equation}
\int |K(x, z, t, s)|dz \le |t-s|^{- \frac{1}{2}} \nu^{-\frac{3}{2} } \norm{(X- \mathrm{Id})}_{Lip(0, T; L^\infty)} C(d,T).
\end{equation}
Similarly, 
\begin{equation}
\begin{gathered}
\int |K(x, z, t, s)| dx = \int (4\pi \nu (t-s))^{- \left (\frac{d}{2} + 1 \right )} \left |  S(\frac{x-X(z,s)}{(4 \nu (t-s))^{\frac{1}{2}}}) - S(\frac{x-X(z,t)}{(4 \nu (t-s))^{\frac{1}{2}}})  \right | dx \\
= \int  \left ( 4 \pi \nu (t-s) \right )^{-1} \pi^{ - \left ( \frac{d}{2} + 1\right )} \left |  S(y) - S( y + \frac{X(z,s)-X(z,t)}{(4 \nu (t-s))^{\frac{1}{2}}})  \right | dy
\end{gathered}
\end{equation}
and again we have
\begin{equation}
\left | \frac{X(z,s)-X(z,t)}{(4 \nu (t-s))^{\frac{1}{2}}} \right | \le \norm{(X-\mathrm{Id})}_{Lip(0, T; L^\infty)} |t-s|^{\frac{1}{2}} \nu^{-\frac{1}{2} } \le C T^{\frac{1}{2}}.
\end{equation}
Therefore, we have the bound
\begin{equation}
\int |K(x,z)|dx \le |t-s|^{- \frac{1}{2}} \nu^{-\frac{3}{2} } \norm{(X- \mathrm{Id})}_{Lip(0, T; L^\infty)} C(d,T).
\end{equation}
\end{proof}
From Lemma \ref{Ktype} and generalized Young's inequality, we have 
\begin{equation}
\begin{gathered}
\norm{ \int_0 ^t \Delta g_{\nu (t-s)} * \Delta_2 \tau (s,t) ds }_{L^p \cap L^\infty}  \le  \frac{C}{\nu}  \left ( \left (\frac{t}{\nu} \right)^{\frac{1}{2}} \norm{X - \mathrm{Id}}_{Lip(0, T; C^{1+\alpha})}  \right ) \norm{\tau}_{L^\infty (0, T; L^p \cap L^\infty )} .
\end{gathered}
\end{equation}
For the H\"{o}lder seminorm, we measure the finite difference. Let us denote $\delta_h f (x, t) = f(x+h, t) - f(x,t)$. If $|h| < t$, then 
\begin{equation}
\delta_h \left ( \int_0 ^t \Delta g_{\nu (t-s)} * \Delta_2 \tau (s,t) ds \right ) = \int_0 ^t \delta_h  (\Delta g_{\nu (t-s) } ) * \Delta_2 \tau (s,t) ds.
\end{equation}
If $0 < t-s < |h|$, then $\norm{\delta_h \Delta g_{\nu (t-s)} }_{L^1} \le 2 \norm{\Delta g_{\nu (t-s)} }_{L^1} \le \frac{C}{\nu (t-s)}$ and since
\begin{equation}
\norm{\Delta_2 \tau (s,t)}_{L^\infty} \le |t-s|^{\alpha} \norm{X - \mathrm{Id}}_{Lip(0, T; C^{1+\alpha})} ^{\alpha} \norm{\tau}_{L^\infty (0, T; C^{\alpha, p})}
\end{equation}
we have
\begin{equation}
\norm{\int_{t-|h|} ^{t} \delta_h (\Delta g_{\nu (t-s)} ) * \Delta_2 \tau (s,t) ds }_{L^\infty } \le \frac{C}{\nu \alpha} |h|^{\alpha} \norm{X - \mathrm{Id}}_{Lip(0, T; C^{1+\alpha})} ^{\alpha} \norm{\tau}_{L^\infty (0, T; C^{\alpha, p})}.
\end{equation}
If $|h| < t-s < t$, then following lines of Lemma \ref{Ktype} $\delta_h (\Delta g_{\nu (t-s)})$ is a $L^1$ function with
\begin{equation}
\norm{\delta_h (\Delta g_{\nu (t-s)})}_{L^1} \le \frac{C |h| } {(\nu (t-s))^{\frac{3}{2}} } \label{Lapga}
\end{equation}
and we have
\begin{equation}
\begin{gathered}
\norm {\int_0 ^{t-|h|} \delta_h (\Delta g_{\nu (t-s)} ) * \Delta_2 \tau (s,t) ds }_{L^\infty }  \\ \le 
\begin{cases}
\frac{C}{\nu^{\frac{3}{2}}}  \norm{X- \mathrm{Id}}_{Lip(0, T; C^{1+\alpha})} ^\alpha \norm{\tau}_{L^\infty (0, T; C^{\alpha, p})} |h|^{\frac{1}{2}} \frac{t^{\alpha}}{\alpha} & \alpha \le \frac{1}{2}, \\
\frac{C}{\nu^{\frac{3}{2}}}  \norm{X- \mathrm{Id}}_{Lip(0, T; C^{1+\alpha})} ^\alpha \norm{\tau}_{L^\infty (0, T; C^{\alpha, p})} |h| \frac{t^{\alpha - \frac{1}{2}}}{\alpha - \frac{1}{2}}  &\alpha > \frac{1}{2}.
\end{cases} 
\end{gathered} 
\end{equation}
If $|h| \ge t$, then we only have the first term. Therefore, we have
\begin{equation}
\begin{gathered}
\frac{1}{|h|^{\alpha} } \norm{\delta_h \left ( \int_0 ^t \Delta g_{\nu (t-s) } * \Delta_2 \tau (s,t) ds \right ) } _{L^\infty}  \le  \frac{C(\alpha)}{\nu} \norm{X- \mathrm{Id}}_{Lip(0, T; C^{1+\alpha})} ^\alpha \norm{\tau}_{L^\infty (0, T; C^{\alpha, p})} .
\end{gathered}
\end{equation}
We note that
\begin{equation}
\begin{gathered}
\norm{\tau (t) }_{\alpha, p} \le \norm{\tau(0) }_{\alpha, p} + t\norm{\tau}_{Lip(0, T; C^{\alpha, p})}. \\
\end{gathered}
\label{tauLip}
\end{equation}
To summarize, we have
\begin{equation}
\begin{gathered}
\norm{\Gamma (\tau \circ X^{-1} ) }_{L^\infty (0, T; C^{\alpha, p})}  
\\ \le C(\alpha) \left (1 + \frac{1}{\nu}  \right ) \norm{X - \mathrm{Id}}_{Lip(0, T; C^{1+\alpha})} ^\alpha \norm{\tau(0) }_{\alpha, p}  + C(\alpha)  \left ( 1 + \frac{1}{\nu} \right ) \norm{X - \mathrm{Id}}_{Lip(0, T; C^{1+\alpha})} ^{\alpha} T \norm{\tau}_{Lip(0,T; C^{\alpha, p})}
\\ + \frac{C(\alpha)}{\nu}  \left (\frac{T}{\nu} \right ) ^{\frac{1}{2}}  \max \{ \norm{X - \mathrm{Id}}_{Lip(0, T; C^{1+\alpha})}^{\alpha}, \norm{X - \mathrm{Id}}_{Lip(0, T; C^{1+\alpha})} ^4 \} (\norm{\tau(0)}_{\alpha, p} + T \norm{\tau}_{Lip(0, T; C^{\alpha, p})} ),
\end{gathered}
\end{equation}
and this completes the proof.
\end{proof}
\begin{theorem}
Let $0 < \alpha < 1, 1 < p < \infty$ and let $T>0$. Let $X' \in Lip (0, T; C^{1+\alpha})$ with $\partial_t X' \in L^\infty (0, T; C^{1+\alpha})$. There exists a constant $C$ such that
\begin{equation}
\begin{gathered}
\norm{ \left [ X' \circ X^{-1} \cdot \nabla, \mathbb{U} \right ] (\sigma) }_{L^\infty (0, T; C^{\alpha, p})} \\
 \le C \left ( \left ( \frac{T} {\nu} \right )^{\frac{1}{2}}  + \frac{T} {\nu}  \norm{X - \mathrm{Id}}_{Lip(0,T; C^{1+\alpha})} \right )M_X ^{1+3\alpha} \norm{X'}_{Lip(0,T;C^{1+\alpha})}   \norm{\sigma}_{L^\infty (0, T; C^{\alpha, p})}
\end{gathered}
\end{equation} \label{commU}
\end{theorem}
\begin{proof}
First, we denote
\begin{equation}
\begin{gathered}
\eta = X' \circ X^{-1}.
\end{gathered}
\end{equation}
Then we have 
\begin{equation}
\begin{gathered}
\left [ \eta \cdot \nabla, \mathbb{U} \right ] (\sigma) (t)
 \\ = \eta(t) \cdot \nabla \int_0 ^t g_{\nu (t-s)} * \mathbb{H} \mathrm{div}\, \sigma(s) ds -\int_0 ^t g_{\nu (t-s)} * \mathbb{H} \mathrm{div} \, (\eta(s) \cdot \nabla \sigma (s) ) ds
 \\ = \left [\eta(t) \cdot \nabla, \mathbb{H} \right ] \int_0 ^t g_{\nu (t-s) } * \mathrm{div}\, \sigma(s) ds + \mathbb{H} \int_0 ^t (\nabla g_{\nu (t-s)} )* \left ( \nabla \cdot \eta (s) \sigma (s) \right ) ds \\
 - \mathbb{H} \int_0 ^t (\nabla \nabla g_{\nu (t-s) } ) * \left ( \eta (s) - \eta (t) \right ) \sigma (s) ds
\\ + \mathbb{H} \int_0 ^t \left ( \eta(t) \cdot (\nabla \nabla g_{\nu (t-s)} ) * \sigma(s) - (\nabla \nabla g_{\nu (t-s) } ) * (\eta(t) \sigma (s) ) \right ) ds,
\end{gathered}
\end{equation}
where $ ( \nabla \nabla g_{\nu (t-s) } ) * (\eta(s) - \eta(t) ) \sigma (s)$,  $\eta(t) \cdot (\nabla \nabla g_{\nu (t-s)} ) * \sigma(s)$, and $(\nabla \nabla g_{\nu (t-s) } ) * (\eta(s) \sigma (s) ) $ represent 
\begin{equation}
\begin{gathered}
\sum_{i,j} (\partial_i \partial_j g_{\nu (t-s) } * ) (\eta_i (s) - \eta_i (t) ) \sigma_{jk } (s),  \\
\sum_{i,j} \eta_i (t) \left ( \partial_i \partial_j g_{\nu (t-s)}  \right )* \sigma_{jk} (s) , \,\, \mathrm{and \;respectively} \sum_{i,j} \left ( \partial_i \partial_j g_{\nu (t-s) } \right ) * (\eta_i (s) \sigma_{jk} (s) ).
\end{gathered}
\end{equation}
The first term is bounded by Lemma \ref{ComEs1} and the second term is estimated directly
\begin{equation}
\begin{gathered}
\norm{ \left [ \eta(t) \cdot \nabla, \mathbb{H} \right ] \int_0 ^t g_{\nu (t-s)} * \mathrm{div} \sigma (s) ds } _{\alpha, p} \le C \norm{\eta (t)}_{C^{1+\alpha}} \left (\frac{t}{\nu} \right )^{\frac{1}{2}} \norm{\sigma}_{L^\infty (0, T; C^{\alpha, p})}, \\
\norm{ \mathbb{H} \int_0 ^t (\nabla g_{\nu(t-s)} ) * (\nabla \cdot \eta (s) \sigma (s) ) ds }_{\alpha, p}  \le C \left ( \frac{t}{\nu} \right )^{\frac{1}{2}} \norm{\eta}_{L^\infty (0, T; C^{1+\alpha})} \norm{\sigma}_{L^\infty (0, T; C^{\alpha, p})}.
\end{gathered}
\end{equation}
The third term is bounded by
\begin{equation}
\frac{Ct}{\nu} \norm{\eta}_{Lip(0, T; C^{\alpha})} \norm{\sigma}_{L^\infty (0, T; C^{\alpha, p})}
\end{equation}
by the virtue of Theorem \ref{p}. For the last term, note that
\begin{equation}
\begin{gathered}
\left (\eta(t) \cdot (\nabla \nabla g_{\nu (t-s) } ) * \sigma (s) - (\nabla \nabla g_{\nu (t-s) } ) * (\eta(t) \sigma (s) ) \right ) (x) 
\\ =\int_{\mathbb{R}^d} \nabla \nabla g_{\nu (t-s)} (z) z \cdot \left ( \int_0 ^1 \nabla \eta ( x - (1-\lambda) z, t) d\lambda \right ) \sigma(x-z, s) dz
\end{gathered} \label{convcom}
\end{equation}
and note that $\nabla \nabla g_{\nu (t-s)} (z) z$ is a $L^1$ function with
\begin{equation}
\norm{\nabla \nabla g_{\nu (t-s)} (z) z }_{L^1} \le \frac{C} {\left (\nu (t-s) \right )^{\frac{1}{2}}}.
\end{equation}
Therefore, 
\begin{equation}
\begin{gathered}
\norm{\left (\eta(t) \cdot (\nabla \nabla g_{\nu (t-s) } ) * \sigma (s) - (\nabla \nabla g_{\nu (t-s) } ) * (\eta(t) \sigma (s) ) \right )}_{\alpha, p} \\ \le \frac{C}{\left (\nu (t-s) \right )^{\frac{1}{2}}} \norm{\eta(t)}_{C^{1+\alpha}} \norm{\sigma(s)}_{\alpha, p}
\end{gathered}
\end{equation}
so that the last term is bounded by 
\begin{equation}
C \left ( \frac{t}{\nu} \right )^{\frac{1}{2}} \norm{\eta(t)}_{C^{1+\alpha}} \norm{\sigma}_{L^\infty (0, T; C^{\alpha, p})}.
\end{equation}
We finish the proof by replacing $\eta$ by $X'$ using Theorem \ref{p}. 
\end{proof}
\begin{theorem}
Let $0<\alpha < 1$, $1 < p < \infty$ and let $T >0$. Let $X' \in Lip(0, T; C^{1+\alpha})$ with $\partial_t X' \in L^\infty (0, T; C^{1+\alpha})$. There exists a constant $C(\alpha)$ depending only on $\alpha$ such that
\begin{equation}
\begin{gathered}
\norm{ \left [X' \circ X^{-1} \cdot \nabla, \mathbb{G} \right ] (\tau \circ X^{-1} ) }_{L^\infty (0, T; C^{\alpha, p})} 
\\ \le ( \norm{X'}_{L^\infty (0, T; C^{1+\alpha})} + \norm{X'}_{Lip(0, T; C^{1+\alpha})} T^{\frac{1}{2}} ) R
\end{gathered}
\end{equation} 
where $R$ is a polynomial function on $\norm{\tau}_{Lip(0, T; C^{\alpha, p})}$, $\norm{X - \mathrm{Id}}_{Lip(0, T; C^{1+\alpha})}$, whose coefficients depend on $\alpha$, $\nu$, and $T$, and in particular it grows polynomially in $T$ and bounded below.
\label{commG}
\end{theorem}
\begin{proof}
Again we denote $\eta = X' \circ X^{-1}$. Also it suffices to bound
\begin{equation}
\begin{gathered}
\left [\eta \cdot \nabla, \Gamma \right ] \left (\tau \circ X^{-1} \right ) = \eta(t) \cdot \nabla \Gamma \left (\tau \circ X^{-1} \right ) - \Gamma \left (\eta \cdot \nabla \left (\tau \circ X^{-1} \right ) \right )
\end{gathered}
\end{equation}
where $\Gamma$ is as defined in (\ref{Gamma}), since
\begin{equation}
\left [ \eta \cdot \nabla, \mathbb{G} \right ] = (R \otimes R ) \mathbb{H} \left [ \eta \cdot \nabla, \Gamma \right ] + \left [ \eta(t) \cdot \nabla, ( R \otimes R ) \mathbb{H} \right ] \Gamma
\end{equation}
and the second term is bounded by Lemma \ref{ComEs1}. For the first term, we have
\begin{equation}
\begin{gathered}
\left [ \eta \cdot \nabla, \Gamma \right ] (\tau \circ X^{-1} ) (t) = I_1 + I_2 + I_3 + I_4 + I_5 + I_6,
\end{gathered}
\end{equation}
where
\begin{equation}
\begin{gathered}
I_1 = \int_0 ^t \eta(t) \cdot \left (\nabla \Delta g_{\nu (t-s) } * \left (\tau \circ X^{-1}(t) \right ) \right ) - \nabla \Delta g_{\nu (t-s) } * \left (\eta(t) \tau \circ X^{-1} (t) \right ) ds, \\
I_2 = \int_0 ^t \eta(t) \cdot \left (\nabla \Delta g_{\nu (t-s) } * \left (\tau \circ X^{-1}(s) - \tau \circ X^{-1}(t) \right ) \right ) \\ - \nabla \Delta g_{\nu (t-s) } * \left (\eta(t) \left (\tau \circ X^{-1}(s) - \tau \circ X^{-1}(t) \right ) \right ) ds, \\
I_3 = - \int_0 ^t \nabla \Delta g_{\nu (t-s) } * \left ( (\eta(s) - \eta(t) ) \left (\tau \circ X^{-1} (s) \right ) \right ) ds, \\
I_4 =  \int_0 ^t \Delta g_{\nu (t-s) } * \left ( \nabla \cdot \left ( \eta(s) -\eta(t) \right ) \tau \circ X^{-1} (s) \right ) ds, \\
I_5 =  \int_0 ^t \Delta g_{\nu (t-s) } * \left ( \nabla \cdot \eta(t) \left (\tau \circ X^{-1} (s) - \tau \circ X^{-1} (t) \right ) \right ) ds, \\
I_6 = -\frac{1}{\nu} \left ( \nabla \cdot \eta (t) \tau \circ X^{-1} (t) - g_{\nu t} * \left ( \nabla \cdot \eta(t) \tau \circ X^{-1} (t) \right ) \right ).
\end{gathered}
\end{equation}
First, $I_1 + I_6$ can be  bounded:
\begin{equation}
\begin{gathered}
I_1 + I_6 = \frac{1}{\nu} \left ( \eta(t) \cdot \nabla \left ( g_{\nu t} * \left ( \tau \circ X^{-1} (t) \right ) \right ) - \nabla \left ( g_{\nu t} * \left (\eta(t) \tau \circ X^{-1} (t) \right ) \right ) \right )  \\ - \frac{1}{\nu} g_{\nu t} * \left (\nabla \cdot \eta(t) \left ( \tau \circ X^{-1} (t) \right ) \right ) 
\end{gathered}
\end{equation}
and the first term is treated in the same way as (\ref{convcom}). Since the first term is
\begin{equation}
\frac{1}{\nu} \left ( \int_{\mathbb{R}^d} \nabla g_{\nu t} (y)  y \cdot \int_0 ^1 \nabla \eta (x - (1-\lambda) y, t) d\lambda \left (\tau \circ X^{-1} \right ) (x-y, t)  dy \right )
\end{equation}
and 
\begin{equation}
\norm{\nabla g_{\nu t} (y) y }_{L^1} \le C,
\end{equation}
the $C^{\alpha, p}$-norm of the first term is bounded by
\begin{equation}
\frac{C}{\nu} \norm{\eta (t) }_{C^{1+\alpha}} \norm{\tau \circ X^{-1} (t) }_{\alpha, p}.
\end{equation}
The $C^{\alpha, p}$-norm of the second term is also bounded by the same bound. Therefore,
\begin{equation}
\norm{I_1  + I_6  }_{L^\infty (0, T; C^{\alpha, p})} \le \frac{C}{\nu} M_X ^{1+3\alpha} \norm{X'}_{L^\infty (0, T; C^{1+\alpha})} \norm{\tau}_{L^\infty (0, T; C^{\alpha, p})}.
\end{equation}
The term $I_3$ is bounded due to Theorem \ref{p}. Since $\eta \in Lip (0, T; C^{\alpha})$ we have
\begin{equation}
\begin{gathered}
\norm{I_3 }_{L^\infty (0, T; C^{\alpha, p})} \le \frac{C}{\nu} \left (\frac{T}{\nu} \right )^{\frac{1}{2}} M_X ^{1+4\alpha} \norm{X - \mathrm{Id}}_{Lip(0, T; C^{1+\alpha})} \\ \norm{X'}_{Lip(0, T; C^{1+\alpha})} \norm{\tau}_{L^\infty (0, T; C^{\alpha, p})}.
\end{gathered}
\end{equation}
The terms $I_4$, and $I_5$ are treated in the spirit of Theorem \ref{G}. We treat $L^p \cap L^\infty$ norm and H\"{o}lder seminorm separately. For the term $I_5$, we have
\begin{equation}
I_5 = \int_0 ^t \Delta g_{\nu (t-s)} * \left ( \nabla \cdot \eta(t) \left ( \Delta_1 \tau (s,t) + \Delta_2 \tau (s,t) \right ) \right ) ds 
\end{equation}
where $\Delta_1 \tau$ and $\Delta_2 \tau$ are the same as (\ref{Deltas}). From the same arguments from the above,
\begin{equation}
\begin{gathered}
\norm{ \int_0 ^t \Delta g_{\nu (t-s)} * \left ( \nabla \cdot \eta(t) \Delta_1 \tau (s,t) \right ) ds}_{\alpha, p} \\ \le \frac{Ct}{\nu} \norm{\eta}_{L^\infty (0, T; C^{1+\alpha})} \norm{\tau}_{Lip(0, T; C^{\alpha, p})} M_X ^{\alpha}.
\end{gathered}
\end{equation}
On the other hand,
\begin{equation}
\begin{gathered}
\Delta g_{\nu (t-s) } * \left ( \nabla \cdot \eta(t) \Delta_2 \tau (s,t) \right ) (x) = \int_{\mathbb{R}^d}  ( K(x,z,t,s) \left ( \nabla \cdot \eta \right ) (X(z,t), t) \\ + \Delta g_{\nu (t-s) } (x - X(z,t) ) \left ( \left (\nabla \cdot \eta  \right ) (X(z,s), t) - \left (\nabla \cdot \eta \right ) (X(z,t), t) \right )  ) dz,
\end{gathered}
\end{equation}
where $K$ is as in (\ref{K}). Then as in the proof of Lemma \ref{Ktype}, by the generalized Young's inequality we have
\begin{equation}
\begin{gathered}
\norm{ \int_0 ^t \Delta g_{\nu (t-s)} * \left ( \nabla \cdot \eta(t) \Delta_2 \tau (s,t) \right ) ds}_{L^p \cap L^\infty} \le C \norm{\tau (t) }_{L^p \cap L^\infty } \norm{\eta}_{L^\infty (0, T; C^{1+\alpha})} \\
\norm{X - \mathrm{Id}}_{Lip(0, T; C^{1+\alpha})} \left ( \frac{t^\alpha}{\nu \alpha} +  \left ( \frac{t}{\nu^3} \right )^{\frac{1}{2}} + \frac{t^2}{\nu^3} \norm{X - \mathrm{Id}}_{Lip(0, T; C^{1+\alpha})} ^3 \right ).
\end{gathered}
\end{equation}
For the H\"{o}lder seminorm, we repeat the same argument in the proof of Theorem \ref{G}, using the bound (\ref{Lapga}). Then we obtain
\begin{equation}
\begin{gathered}
\frac{1}{|h|^{\alpha} } \norm{\delta_h \left ( \int_0 ^t \Delta g_{\nu (t-s) } * \Delta_2 \tau (s,t) ds \right ) } _{L^\infty} \\ 
\le  \frac{C(\alpha)}{\nu} \left ( 1 + \left (\frac{t}{\nu} \right )^{\frac{1}{2}} + \left (\frac{t}{\nu} \right )^2 \right )\norm{X- \mathrm{Id}}_{Lip(0, T; C^{1+\alpha})} ^\alpha \norm{\tau}_{L^\infty (0, T; C^{\alpha, p})} \norm{\eta}_{L^\infty (0, T; C^{1+\alpha})} .
\end{gathered}
\end{equation}
Therefore,
\begin{equation}
\begin{gathered}
\norm{I_5}_{L^\infty (0 , T; C^{\alpha, p})}  \le \frac{C(\alpha)}{\nu} \left (1 + t + \left ( \frac{t}{\nu} \right )^2 \right )  \left (1 + \norm{X - \mathrm{Id} }_{Lip(0, T; C^{1+\alpha})} \right )^3 M_X ^{1+2\alpha}
\\ \norm{X'}_{L^\infty (0, T; C^{1+\alpha})} \norm{\tau}_{Lip(0, T; C^{\alpha, p})}.
\end{gathered}
\end{equation}
The term $I_4 (t)$ is treated in the exactly same way, by noting that
\begin{equation}
\begin{gathered}
\nabla \cdot \left (\eta(s) - \eta(t) \right ) = \nabla_x X^{-1} (s) : (\Delta_1 \nabla_a X' (s,t) ) + \nabla_x X^{-1} (s) : (\Delta_2 \nabla_a X' (s,t) )\\
 + \left (\nabla_x X^{-1} (s) - \nabla_x X^{-1} (t) \right ) : \left ( \nabla_a X' \circ X^{-1} \right ) (t),
\end{gathered}
\end{equation}
where as in (\ref{Deltas})
\begin{equation}
\begin{gathered}
\Delta_1 \nabla_a X' (x,s,t) = \nabla_a X' (X^{-1} (x,s), s) - \nabla_a X' (X^{-1} (x,s), t), \\
\Delta_2 \nabla_a X' (x,s,t) = \nabla_a X' (X^{-1} (x,s), t) - \nabla_a X' (X^{-1} (x,t), t),
\end{gathered}
\end{equation}
and
\begin{equation}
\nabla_x \left ( X^{-1} (x,s) - X^{-1} (x,t) \right )= \left ( \nabla_a X \circ X^{-1} \right ) (x, t) \left ( \nabla_a \left (X - \mathrm{Id} \right ) \right ) \left ( X^{-1} (x,t), t-s \right )
\end{equation}
so that 
\begin{equation}
\begin{gathered}
\norm{\nabla_x X^{-1} (s) - \nabla_x X^{-1} (t) }_{C^{\alpha} } \le |t-s| \norm{X - \mathrm{Id}}_{Lip(0, T; C^{1+\alpha})}M_X ^{1+ 2\alpha}.
\end{gathered}
\end{equation}
Also note that
\begin{equation}
\norm{\Delta_2 \nabla_a X'(s,t) }_{L^\infty } \le \norm {\nabla_a X'(t) }_{C^{\alpha}} \norm{X - \mathrm{Id}}_{Lip(0, T; L^\infty)}^{\alpha} |t-s|^{\alpha}
\end{equation}
so that
\begin{equation}
\begin{gathered}
\norm{\int_0 ^t \Delta g_{\nu (t-s) } * \left ( \nabla_x X^{-1} (s) : (\Delta_2 \nabla_a X' (s,t) ) \tau \circ X^{-1} (s) \right ) ds}_{C^{\alpha, p}}
\\ \le \frac{C(\alpha)} {\nu} \left ( 1 + t^{\alpha} + \left (\frac{t}{\nu} \right )^2 \right ) M_X ^{1+2 \alpha} \norm{X - \mathrm{Id}}_{Lip(0, T; C^{1+\alpha})} ^{\alpha}
\\  \norm{X'}_{L^\infty (0, T; C^{1+\alpha})} \norm{\tau}_{L^\infty (0, T; C^{\alpha, p})}. 
\end{gathered}
\end{equation}
The final result is
\begin{equation}
\begin{gathered}
\norm{I_4 (t) }_{\alpha, p} \le \frac{C(\alpha)}{\nu} \left (1 + t + \left (\frac{t}{\nu} \right )^2 \right ) M_X ^{2+ 4\alpha} 
 \norm{X'}_{L^\infty (0, T; C^{1+\alpha})} \norm{\tau}_{L^\infty (0, T; C^{\alpha, p})}
 \\ + C \frac{t}{\nu} M_X ^{1+3\alpha} \norm{X'}_{Lip(0, T; C^{1+\alpha})} \norm{\tau}_{L^\infty (0, T; C^{\alpha, p})}.
\end{gathered}
\end{equation}
Finally, $I_2$ can be bounded using the combination of the technique in Theorem \ref{G} and Theorem \ref{commU}. First, we have
\begin{equation}
\begin{gathered}
I_2 (x,t)  = \\  \int_0 ^t \int_{\mathbb{R}^d} \nabla \Delta g_{\nu (t-s) } (y) \cdot y \cdot  \left ( \int_0 ^1 \nabla \eta (x - (1 - \lambda)y, t ) d\lambda \left ( \Delta_1 \tau (x-y, s,t) \right ) \right ) dy ds
\\ + \int_0 ^t \int_{\mathbb{R}^d} \nabla \Delta g_{\nu (t-s) } (x-z) \cdot (x-z) \cdot  \left ( \int_0 ^1 \nabla \eta (\lambda x + (1-\lambda) z, t ) d\lambda \left ( \Delta_2 \tau (z, s,t) \right ) \right ) dz ds.
\end{gathered}
\end{equation}
Then applying the argument of the proof of Theorem \ref{commU}, the first term is bounded by
\begin{equation}
\frac{C}{\nu} t M_X ^{\alpha} \norm{\eta}_{L^\infty (0, T; C^{1+\alpha})} \norm{\tau}_{Lip(0, T; C^{\alpha, p})}.
\end{equation}
The second term is treated using the method used in Theorem \ref{G}. By changing variables to form a kernel similar to (\ref{K}), and applying generalized Young's inequality, the $L^p \cap L^\infty$ norm of the second term is bounded by
\begin{equation}
\frac{C(\alpha)}{\nu} \left ( t^{\alpha} + \left (\frac{t}{\nu} \right )^{\frac{1}{2}} + \left (\frac{t}{\nu} \right )^{2} \right ) \left (1 + \norm{X - \mathrm{Id}}_{Lip(0, T; C^{1+\alpha})} \right )^4 \norm{\eta}_{L^\infty (0, T; C^{1+\alpha})} \norm{\tau}_{L^\infty (0, T; L^p \cap L^\infty ) }.
\end{equation}
Finally, the H\"{o}lder seminorm of the second term is bounded by the same method as Theorem \ref{G}. The only additional point is the finite difference of $\nabla \eta$ term, but this term is bounded by a straightforward estimate. The bound for the H\"{o}lder seminorm of the second term is
\begin{equation}
\frac{C(\alpha)}{\nu} \left (1 + t^{\alpha} + \left (\frac{t}{\nu} \right )^{\frac{1}{2}} + \left (\frac{t}{\nu} \right )^2 \right ) \norm {X - \mathrm{Id}}_{Lip(0, T; C^{1+\alpha})} ^{\alpha} \norm{\eta}_{L^\infty (0, T; C^{1+\alpha})} \norm{\tau}_{L^\infty (0, T; C^{\alpha, p})}.
\end{equation}
To sum up, we have
\begin{equation}
\begin{gathered}
\norm{I_2 (t)}_{\alpha, p} \le \frac{C(\alpha)}{\nu} \left (1 + t + \left (\frac{t}{\nu} \right )^2 \right ) \left (1 + \norm {X- \mathrm{Id}}_{Lip(0, T; C^{1+\alpha})} \right )^{4} M_X ^{1+3\alpha}
\\  \norm{X'}_{L^\infty (0, T; C^{1+\alpha})} \norm{\tau}_{Lip(0, T; C^{\alpha, p})}.
\end{gathered}
\end{equation}
If we put this together, 
\begin{equation}
\begin{gathered}
\norm{ \left [X' \circ X^{-1} \cdot \nabla, \mathbb{G} \right ] (\tau \circ X^{-1} ) }_{L^\infty (0, T; C^{\alpha, p})} 
\\ \le C \norm{X'}_{L^\infty (0 ,T; C^{1+\alpha})} M_X ^{1+2\alpha} \norm{\Gamma (\tau \circ X^{-1}) }_{L^\infty (0, T; C^{\alpha, p})}
\\ +  ( \norm{X'}_{L^\infty (0, T; C^{1+\alpha})}  + \norm{X'}_{Lip(0, T; C^{1+\alpha})} T^{\frac{1}{2}} ) F_1 (\nu, \alpha, X, \norm{\tau}_{Lip(0, T; C^{\alpha, p})}, T ) 
\end{gathered}
\end{equation} 
where $F_1$ depends on the written variables and grows like polynomial in $T, \norm{\tau}_{Lip(0, T; C^{\alpha, p})}$, and $\norm{X - \mathrm{Id}}_{Lip(0, T; C^{1+\alpha})}$. The bound on $\Gamma (\tau \circ X^{-1})$ is given by Theorem \ref{G}.
\end{proof}


\section{Bounds on variations and variables} \label{Bounds}
Using the results from the previous section we find bounds for variations and variables. For simplicity, we adopt the notation
\begin{equation}
M_\epsilon = 1 + \norm{X_\epsilon - \mathrm{Id} }_{L^\infty (0, T; C^{1+\alpha})}.
\end{equation}
First, we bound  $\frac{d}{d\epsilon} \mathcal{V}_\epsilon$. Note that $X_\epsilon (0) = \mathrm{Id}$, so $X_\epsilon ' (0) = 0$ and by Theorem \ref{p} and since $X_\epsilon ' \in Lip(0, T; C^{1+\alpha, p})$  we have
\begin{equation}
\begin{gathered}
\norm{X_\epsilon'}_{L^\infty (0, T; C^{1+\alpha})} \le T \norm{X_\epsilon '}_{Lip(0, T; C^{1+\alpha, p})},
\\ \norm{\eta_\epsilon (t)}_{C^{\alpha}} \le t \norm{X'}_{Lip(0, T; C^{1+\alpha, p})} M_\epsilon ^{\alpha}.
\end{gathered}
\end{equation}
Then by the Theorem \ref{L}, we have
\begin{equation}
\begin{gathered}
\norm{ \eta_\epsilon \cdot \mathbb{L}_\nu (\nabla_x u_{\epsilon, 0} ) }_{L^\infty (0, T; C^{\alpha, p})} \le C \left (\frac{T}{\nu} \right )^{\frac{1}{2}} M_\epsilon ^{\alpha} \norm{X_\epsilon '}_{Lip(0, T; C^{1+\alpha, p})} \norm{u_{\epsilon, 0}}_{1+\alpha, p}, \\
\norm{\mathbb{L}_{\nu} (u_{\epsilon, 0} ' )  }_{L^\infty (0, T; C^{\alpha, p})} \le C \norm{u_{\epsilon, 0} '}_{\alpha, p}.
\end{gathered}
\end{equation}
By Theorem \ref{commU}, we have
\begin{equation}
\begin{gathered}
\norm{ \left [ \eta_\epsilon \cdot \nabla_x, \mathbb{U} \right ] (\sigma_\epsilon - u_\epsilon \otimes u_\epsilon ) }_{L^\infty (0, T; C^{\alpha, p})} \le C \left ( \left (\frac{T}{\nu} \right )^{\frac{1}{2}} +\left (\frac{T}{\nu} \right )\right )M_\epsilon ^{2+4\alpha}
\\  \norm{X_\epsilon '}_{Lip(0, T; C^{1+\alpha})} \norm{\tau_\epsilon - v_\epsilon \otimes v_\epsilon }_{L^\infty (0, T; C^{\alpha, p})},
\end{gathered}
\end{equation}
and by Theorem \ref{U}, we have
\begin{equation}
\begin{gathered}
\norm{\mathbb{U} \left (\delta_\epsilon - (v_\epsilon ' \otimes v_\epsilon + v_\epsilon \otimes v_\epsilon ' ) \circ X_\epsilon ^{-1} \right )}_{L^\infty (0, T; C^{\alpha, p})} \le C \left (\frac{T}{\nu } \right )^{\frac{1}{2}} M_\epsilon ^{\alpha} 
\\ \norm{\tau_\epsilon ' - (v_\epsilon ' \otimes v_\epsilon + v_\epsilon \otimes v_\epsilon ' )}_{L^\infty (0, T; C^{\alpha, p})}.
\end{gathered}
\end{equation}
Therefore, 
\begin{equation}
\begin{gathered}
\norm{\frac{d}{d\epsilon} \mathcal{V_\epsilon} }_{L^\infty (0, T; C^{\alpha, p})} \le C \norm{u_{\epsilon, 0} '}_{\alpha, p} 
\\ + S_1 (T)  (\norm{X_\epsilon '}_{Lip(0,T;C^{1+\alpha, p})} + \norm{v_\epsilon '}_{L^\infty (0, T; C^{\alpha, p})} + \norm{\sigma_{\epsilon,0} ' }_{\alpha, p} + \norm{\tau_\epsilon '}_{Lip(0, T; C^{\alpha, p})} ) Q_1
\end{gathered} 
\end{equation}
where $S_1(T) $ vanishes as $T^{\frac{1}{2}}$ as $T \rightarrow 0$ and $Q_1$ is a polynomial in $\norm{u_{\epsilon, 0}}_{1+\alpha, p}$, $\norm{X_\epsilon - \mathrm{Id}}_{Lip(0, T; C^{1+\alpha, p})}$, $\norm{\tau_\epsilon}_{L^\infty (0, T; C^{\alpha, p})}$, and $\norm{v_\epsilon}_{L^\infty (0, T; C^{\alpha, p})}$, whose coefficients depend on $\nu$.  Similarly,
\begin{equation}
\begin{gathered}
\norm{g_\epsilon}_{L^\infty (0, T; C^{\alpha, p})} \le M_X ^{\alpha} \norm{u_0}_{1+\alpha, p} + C_1 \norm{X - \mathrm{Id}}_{Lip(0, T; C^{1+\alpha})} ^{\alpha} \norm{\sigma_{\epsilon,0}}_{\alpha, p} + S_2 (T) Q_2,
\end{gathered}
\end{equation}
where $S_2 (T) $ vanishes as $T^{\frac{1}{2}}$ as $T \rightarrow 0$ and $Q_2$ is polynomial in $\norm{\tau}_{Lip(0,T; C^{\alpha, p})}$ and $\norm{X - \mathrm{Id}}_{Lip(0, T; C^{1+\alpha})}$, whose coefficients depend on $\alpha$ and $\nu$. Also
\begin{equation}
\begin{gathered}
\norm{g_\epsilon '}_{L^\infty (0, T; C^{\alpha, p})} \le C (\norm{u_{\epsilon, 0}'}_{1+\alpha, p} + \norm{X - \mathrm{Id}}_{Lip(0, T; C^{1+\alpha})} ^{\alpha} \norm{\tau_{\epsilon, 0} '}_{\alpha, p} ) 
\\ + S_3 (T) (\norm{X_\epsilon '}_{Lip(0,T; C^{1+\alpha, p})} + \norm{\sigma_{\epsilon, 0} ' }_{\alpha, p} + \norm{\tau_\epsilon '}_{Lip(0, T; C^{\alpha, p})} + \norm{v_\epsilon '}_{L^\infty (0, T; C^{1+\alpha, p})} )Q_3,
\end{gathered}
\end{equation}
where $S_3 (T)$ vanishes as $T^{\frac{1}{2}}$ as $T \rightarrow 0$ and $Q_3$ is polynomial in $\norm{u_{\epsilon, 0}}_{1+\alpha, p}$, $\norm{X - \mathrm{Id}}_{Lip(0, T; C^{1+\alpha, p}}$, $\norm{\tau}_{Lip(0, T; C^{\alpha, p})}$, and $\norm{v_\epsilon}_{L^\infty (0, T; C^{1+ \alpha, p})}$, whose coefficients depend on $\nu$ and $\alpha$.  Then we have
\begin{equation}
\norm{ \nabla_a \frac{d}{d\epsilon} \mathcal{V}_\epsilon }_{L^\infty (0, T; C^{\alpha, p})} \le T \norm{X_\epsilon '}_{Lip (0, T; C^{1+\alpha})} \norm{g_\epsilon}_{L^\infty (0, T; C^{\alpha, p})} + M_\epsilon \norm{g_\epsilon '}_{L^\infty (0, T; C^{\alpha, p})}
\end{equation}
and
\begin{equation}
\begin{gathered}
\norm{\frac{d}{d\epsilon } \mathcal{T}_\epsilon }_{L^\infty (0, T; C^{\alpha, p})} \le 2 \norm {g_\epsilon ' }_{L^\infty (0, T; C^{\alpha, p})} \left (\norm{\tau_\epsilon }_{L^\infty (0, T; C^{\alpha, p})}  + 2 \rho K \right )
\\ + \norm{\tau_\epsilon '}_{L^\infty (0, T; C^{\alpha, p})} \left ( \norm{g_\epsilon}_{L^\infty (0, T; C^{\alpha, p})} + 2k \right ).
\end{gathered}
\end{equation}


\section{Local existence}\label{Localex}
We define the function space $\mathcal{P}_1$ and the set $\mathcal{I}$,
\begin{equation}
\begin{gathered}
\mathcal{P}_1 = Lip(0, T; C^{1+\alpha, p}) \times Lip(0, T; C^{\alpha, p}) \times L^\infty (0, T; C^{1+\alpha, p}) \\
\mathcal{I} = \{ (X, \tau, v) : \norm{(X - \mathrm{Id}, \tau, v)}_{\mathcal{P}_1} \le \Gamma, v = \frac{dX}{dt} \},
\end{gathered}
\end{equation} 
where $\Gamma > 0$ and $T>0$ are to be determined. 
Now, for given $u_0 \in C^{1+\alpha, p}$ divergence free and $\sigma_0 \in C^{\alpha, p}$ we define the map
\begin{equation}
(X, \tau, v) \rightarrow \mathcal{S} (X, \tau, v) = (X^{new}, \tau^{new}, v^{new} )
\end{equation}
where
\begin{equation}
\left \{
\begin{gathered}
X^{new} (t) = \mathrm{Id} + \int_0 ^t \mathcal{V} (X(s), \tau(s), v(s) ) ds, \\
\tau ^{new} (t) = \sigma_0 + \int_0 ^t \mathcal{T} (X(s), \tau(s), v(s)) ds, \\
v^{new} (t) = \mathcal{V} (X, \tau, v).
\end{gathered} \right. \label{Solmap}
\end{equation}
If $(X - \mathrm{Id}, \tau, v) \in \mathcal{P}_1$, then $(X^{new} - \mathrm{Id}, \tau^{new}, v^{new} ) \in \mathcal{P}_1$ for any choice of $T>0$. Moreover, we have the following:
\begin{theorem}
For given $u_0 \in C^{1+\alpha, p}$ divergence free and $\sigma_0 \in C^{\alpha, p}$, there is a $\Gamma > 0$ and $T>0$ such that the map $\mathcal{S}$ of (\ref{Solmap}) maps $\mathcal{I}$ to itself. \label{itself}
\end{theorem}
\begin{proof}
It is obvious that $\frac{d}{dt} X^{new} = v^{new}$. For the size of $\mathcal{S} (X, \tau, v)$, first note that if $(X - \mathrm{Id}, \tau, v)_{\mathcal{P}_1 } \le \Gamma$, then
\begin{equation}
M_X = 1 + \norm{X - \mathrm{Id}}_{L^\infty (0, T; C^{1+\alpha})} \le 1 + T \Gamma.
\end{equation}
Applying Theorem \ref{L} and Theorem \ref{U}, we know that
\begin{equation}
\norm{\mathcal{V}}_{L^\infty (0, T; C^{\alpha, p})} \le \norm{u_0}_{\alpha, p} + A_1 (T) B_1 (\Gamma, \norm{u_0}_{\alpha, p}, \norm{\sigma_0}_{\alpha, p}), 
\end{equation}
where $A_1 (T)$ vanishes like $T^{\frac{1}{2}}$ for small $T>0$ and $B_1$ is a polynomial in its arguments, and some coefficients depend on $\nu$.
We estimate
\begin{equation}
\norm{g}_{L^\infty (0, T; C^{\alpha, p})} \le \norm{u_0}_{1+\alpha, p} + C_1 \Gamma ^{\alpha} \norm{\sigma_0}_{\alpha, p} + A_2 (T) B_2 (\Gamma, \norm{u_0}_{1+\alpha, p}, \norm{\sigma_0}_{\alpha,p}),
\end{equation}
where $C_1$ is as in Theorem \ref{G}, depending only on $\alpha$ and $\nu$, $A_2 (T)$ vanishes in the same order as $A_1 (T)$ as $T \rightarrow 0$, and $B_2$ is a polynomial in its arguments, and some coefficients depend on $\nu$ and $\alpha$.
From (\ref{Vg}) we conclude
\begin{equation}
\norm{\mathcal{V}}_{L^\infty (0, T; C^{1+\alpha, p})} \le K_1 ( \norm{u_0}_{1+ \alpha, p} + \Gamma ^{\alpha} \norm{\sigma_0}_{\alpha, p} ) + A_3 (T) B_3(\Gamma, \norm{u_0}_{1+\alpha, p}, \norm{\sigma_0}_{\alpha, p}) \label{itself1},
\end{equation}
where $K_1$ is a constant depending only on $\nu$ and $\alpha$, and $A_3$ and $B_3$ have the same properties as previous $A_i$s and $B_i$s. Now we measure $\mathcal{T}$. From (\ref{tauLip}) and the previous estimate on $g$ we have
\begin{equation}
\begin{gathered}
\norm{\mathcal{T}}_{L^\infty (0, T; C^{\alpha, p})} \le K_2 (\norm{u_0}_{1+\alpha, p} ( \rho K + \norm{\sigma_0}_{\alpha, p}) + \norm{\sigma_0}_{\alpha, p} ( \Gamma^{\alpha} \norm{\sigma_0}_{\alpha, p}  + \rho K \Gamma ^{\alpha} +k ) ) 
\\ + A_4 B_4 \label{itself2},
\end{gathered}
\end{equation}
where $K_2$ is a constant depending on $\nu$ and $\alpha$, and $A_4$ and $B_4$ are as before. Since $\alpha<1$, we can appropriately choose large $\Gamma > \norm{\sigma_0}_{\alpha, p} + \norm{u_0}_{1+\alpha, p}$ and correspondingly small $\frac{1}{6} >T>0$ so that the right side of (\ref{itself1}) and (\ref{itself2}) are bounded by $\frac{\Gamma}{6}$. Then $\norm{ (X^{new} - \mathrm{Id}, \tau^{new}, v^{new} ) }_{\mathcal{P}_1} \le \Gamma$.
\end{proof}
We show now that $\mathcal{S}$ is a contraction mapping on $\mathcal{I}$ for a short time.
\begin{theorem}
For given $u_0 \in C^{1+\alpha, p}$ divergence free and $\sigma_0 \in C^{\alpha, p}$, there is a $\Gamma$ and $T>0$, depending only on $\norm{u_0}_{1+\alpha, p}$ and $\norm{\sigma_0}_{\alpha, p}$, such that the map $\mathcal{S}$ is a contraction mapping on $\mathcal{I} = \mathcal{I}(\Gamma, T)$, that is
\begin{equation}
\norm{ \mathcal{S} (X_2, \tau_2, v_2 ) - \mathcal{S} (X_1, \tau_1, v_1) }_{\mathcal{P}_1} \le \frac{1}{2} \norm{(X_2 - X_1, \tau_2 - \tau_1, v_2 - v_1 ) }_{\mathcal{P}_1}.
\end{equation} \label{contraction}
\end{theorem}
\begin{proof}
First from Theorem \ref{itself} we can find a $\Gamma$ and $T_0 >0$, depending only on the size of initial data, say
\begin{equation}
N = \max \{ \norm{u_0}_{1+\alpha, p}, \norm{\sigma_0}_{\alpha, p} \},
\end{equation}
 which guarantees that $\mathcal{S}$ maps $\mathcal{I}$ to itself. This property still holds if we replace $T_0$ by any smaller $T>0$. In view of the fact that $\mathcal{I}$ is convex, we put
\begin{equation}
\begin{gathered}
X_\epsilon = (2-\epsilon) X_1 + (\epsilon -1) X_2, \\
\tau_\epsilon = (2- \epsilon) \tau_1 + (\epsilon -1 ) \tau_2, 1 \le \epsilon \le 2.
\end{gathered}
\end{equation}
Then $(X_\epsilon, \tau_\epsilon, v_\epsilon ) \in \mathcal{I}$, $v_\epsilon = (2-\epsilon) v_1 + (\epsilon -1) v_2$, $u_{\epsilon,0} = u_0$, and $\sigma_{\epsilon, 0} = \sigma_0$. This means that 
\begin{equation}
X_\epsilon ' = X_2 - X_1, v_\epsilon ' = v_2 - v_1, u_{\epsilon, 0} ' = 0, \sigma_{\epsilon, 0} ' = 0.
\end{equation}
Then from the results of Section \ref{Bounds}, we see that
\begin{equation}
\begin{gathered}
\norm{\frac{d}{d\epsilon} \mathcal{V}_\epsilon}_{L^\infty (0, T; C^{1+\alpha, p})}   \le (\norm{X_2 - X_1 }_{Lip(0, T; C^{1+\alpha, p})} + \norm{v_2 - v_1}_{L^\infty (0, T; C^{\alpha, p})} \\+ \norm{\tau_2 - \tau_1 }_{Lip(0, T; C^{\alpha, p})} ) S_ 1 ' (T) Q_1 ' (\Gamma), \\
\norm{\mathcal{X}_\epsilon '}_{Lip (0, T; C^{1+\alpha, p})}  \le (\norm{X_2 - X_1 }_{Lip(0, T; C^{1+\alpha, p})} + \norm{v_2 - v_1}_{L^\infty (0, T; C^{\alpha, p})} \\+ \norm{\tau_2 - \tau_1 }_{Lip(0, T; C^{\alpha, p})} ) S_ 2 ' (T) Q_2 ' (\Gamma), \\
\\ \norm{\pi_\epsilon} _{Lip(0, T; C^{\alpha, p})}  \le (\norm{X_2 - X_1 }_{Lip(0, T; C^{1+\alpha, p})} + \norm{v_2 - v_1}_{L^\infty (0, T; C^{\alpha, p})} \\+ \norm{\tau_2 - \tau_1 }_{Lip(0, T; C^{\alpha, p})} ) S_ 3 ' (T) Q_3 ' (\Gamma),
\end{gathered}
\end{equation}
where $\mathcal{X}_\epsilon '$ and $\pi_\epsilon$ are defined in (\ref{vars}), $S_1 '(T), S_2 '(T), S_3 '(T)$ vanish at the rate of $T^{\frac{1}{2}}$ as $T \rightarrow 0$, and $Q_1 ' (\Gamma), Q_2 ' (\Gamma), Q_3 '(\Gamma) $ are polynomials in $\Gamma$, whose coefficients depend only on $\nu$ and $\alpha$. By choosing $0 <T < T_0$ small enough, depending on the size of $Q_i ' (\Gamma)$s, we conclude the proof.
\end{proof}
We have obtained a solution to the system (\ref{sys}) in the path space $\mathcal{P}_1$ for a short time, that is, we have $(X, \tau, v)$ satisfying $v = \frac{dX}{dt}$ and satisfying (\ref{fixedpt}). We also have Lipschitz dependence on initial data, Theorem {\ref{main}}.

\begin{proof}
We repeat the calculation of the Theorem \ref{contraction}, but this time $u_{\epsilon, 0}' = u_1 (0) - u_2 (0)$ and $\sigma_{\epsilon, 0} ' = \sigma_1 (0) - \sigma_2 (0)$. Then we choose $T_0$ small enough that $S_i '(T_0) Q_1 ' (\Gamma) < \frac{1}{2}$.
\end{proof}


\paragraph{Acknowledgements} The research of P.C. was partially supported by NSF grant DMS-1713985. The research of J.L. was partially supported by a Samsung scholarship.

\bibliographystyle{abbrv}

\end{document}